\documentclass{article}

\usepackage{arxiv}
\usepackage{amsmath}
\usepackage{amssymb}
 \usepackage{amscd}

\usepackage[utf8]{inputenc} 
\usepackage[T1]{fontenc}    
\usepackage{url}            
\usepackage{booktabs}       
\usepackage{amsfonts}       
\usepackage{nicefrac}       
\usepackage{microtype}      
\usepackage{graphicx}
\usepackage{doi}

\newtheorem{theorem}{Theorem}[section]

\newtheorem{example}{Example}
\newtheorem{lemma}[theorem]{Lemma}

\newtheorem{remark}[theorem]{Remark}

\newenvironment{proof}[1][Proof]{\textbf{#1.} }{\ \rule{0.5em}{0.5em}}

\title{On the doubly stochastic realization of spectra}

\author{ Kassem Rammal \\
	KALMA, Department of Mathematics,\\  Faculty of Science, Lebanese University,\\ Beirut, Lebanon
	\texttt{kassem.rammal@hotmail.com} \\
	\And
	Bassam Mourad \\
	KALMA, Department of Mathematics,\\  Faculty of Science, Lebanese University,\\ Beirut, Lebanon
	\texttt{ bmourad@ul.edu.lb} \\
    \And
	Hassan Abbas \\
	KALMA, Department of Mathematics,\\  Faculty of Science, Lebanese University,\\ Beirut, Lebanon
	\texttt{habbas@ul.edu.lb} \\
   \And
	Hassan Issa \\
	KALMA, Department of Mathematics,\\  Faculty of Science, Lebanese University,\\ Beirut, Lebanon
	\texttt{hissa@uni-math.gwdg.de} \\
}

\begin{document}
\maketitle

\begin{abstract}
An $n$-list $\lambda = \left(r; \lambda_2, \ldots, \lambda_n\right)$ of complex numbers with $r>\max\limits_{2\leq j\leq n}|\lambda_j|,$ is said to be realizable if $\lambda$ is the spectrum of $n\times n$ nonnegative matrix $A$ and in this case $A$ is said to be a nonnegative realization of $\lambda$.  If, in addition,  each row and column sum  of $A$  is equal to $r$, then $\lambda$ is said to be doubly stochastically realizable and in such case $A$ is said to be a doubly stochastic realization for $\lambda$.
In 1997, Guo  proved that if $\left(\lambda_2,\ldots, \lambda_n\right)$ is any list of complex numbers which is closed under complex conjugation then there exists a least real number $\lambda_0$
with $\max\limits_{2\leq j\leq n}|\lambda_j|\leq\lambda_0\leq 2n\max\limits_{2\leq j\leq n}|\lambda_j|$
such that the list of complex numbers $\left( \lambda_1,\lambda_2,\ldots,\lambda_n\right)$ is realizable if and only if $\lambda_1\geq \lambda_0$.
Many researchers deal with sharpening this upper bound   and others  are concerned with finding the exact value of $\lambda_0$ for  particular classes of matrices\cite{enide,chris,soto_20,maria}.

In this paper, we first describe a method for passing from a nonnegative realization to a doubly stochastic realization. As  applications, we give a new sufficient condition for a stochastic matrix $A$ to be cospectral to a doubly stochastic matrix $B$ and in this case $B$ is shown to be the unique closest doubly stochastic matrix to $A$ with respect to the Frobenius norm. Then, our next result
 gives an improvement of  Guo's bound which also sharpens the existing known bound for the case when one of at least one of  $ \left\{ \lambda_2, \ldots, \lambda_n\right\}$
  is real. Furthermore,  we investigate the case when  $ \left\{ \lambda_2, \ldots, \lambda_n\right\}$  are all non-real which has not been dealt before. Our main results here also  sharpen  Guo's bound. Next, for doubly stochastic realizations, we obtain an upper bound that improves Guo's bound as well. Finally, for certain particular cases, we give a further improvement of our last bound for doubly stochastic realization.
\end{abstract}

\keywords{nonnegative matrices, doubly stochastic matrices, inverse eigenvalue problem, nonnegative realization, doubly stochastic realization}

\section{Introduction} 
An $m\times n$ matrix $A$ with real entries is said to be \emph{nonnegative} if all of its entries are nonnegative. If in addition, each row sum of $A$ is equal to $1,$ then $A$ is called \emph{stochastic} (or \emph{row-stochastic}). 
Generally, an  $n\times n$  matrix $A$ over the field of the real numbers $\mathbb{R}$ having each row sum equals to a nonnegative number $r \in \mathbb{R}^+,$ is said to be an $r-$\emph{generalized stochastic} matrix (note that $A$ is not necessarily nonnegative). If $A$ and its transpose $A^T$ are $r-$generalized stochastic matrices then $A$ is said to be an $r-$\emph{generalized doubly stochastic} matrix. The set of all $r-$generalized  $n\times n$  doubly stochastic matrices with entries in $\mathbb{R}$ is denoted by $\Omega^r(n).$ An $n\times n$ generalized doubly stochastic matrix is an element
$\Omega(n)$ where
\begin{equation*}
	\Omega(n)=\bigcup_{r\in \mathbb{R}^+}\Omega^r(n).
\end{equation*}
Of special importance are the nonnegative elements in $\Omega(n)$ and in particular the nonnegative elements in $\Omega^1(n)$ which are called the \emph{doubly stochastic} matrices and have been the object of study for a long time see \cite{John_81,Mourad_03,MouradALL_13,Mourad_15,Perferct_65,Soules_83}.

The Perron-Frobenius theorem states that if $A$ is a nonnegative matrix, then it has a nonnegative eigenvalue $r$ (that is the Perron root) which is greater than or equal to the modulus of each of the other eigenvalues (see, e.g., \cite{Minc_88}). To this eigenvalue $r$ of $A$ corresponds a nonnegative eigenvector $x$ which is also referred to as the Perron-Frobenius eigenvector of $A.$ In particular, it is well-known that if $A$ is an $n\times n$ stochastic matrix then its Perron eigenvalue $r = 1$ and its corresponding unit eigenvector is the column vector $x = e = \left(1,\ldots, 1\right)^T \in \mathbb{R}^n.$ Consequently, this is also true for $A$ and $A^T$ when $A$ is doubly stochastic. More generally, for any $X\in \Omega^r(n)$,  $e$ is clearly an eigenvector for both $X$ and $X^T$ corresponding to the eigenvalue $r.$ Therefore, $X \in \Omega^r(n)$ if and only if $Xe= re$ and $e^T X = re^T$ but this in turn is equivalent to $XJ_n = J_nX = rJ_n,$ where $J_n$ is the  $n\times n$  matrix with each of its entries is equal to $\frac{1}{n}.$

Two matrices $A$ and $B$ are said to be \emph{cospectral} if they have the same set of eigenvalues. Throughout this paper, if $A = \left(a_{ij}\right)$ is any square matrix, then the \emph{spectrum} of $A$ is denoted by $\sigma(A)$ and when $A$ is real, then $s_j(A)$ and $a_j$ denote the \emph{sum} and the \emph{smallest entry} of the $j$th column of $A$ respectively. For any complex number $z$, its \emph{real part} is denoted by $\Re(z)$ and its \emph{imaginary part} is denoted by $\Im(z)$. Moreover, for any real number $a,$ its \emph{absolute value} will be denoted by $|a|,$ and the  $n\times n$  \emph{identity} matrix will be denoted by $I_n.$ Next, we shall borrow the following notation which first appeared in \cite{Fiedler_74}. To indicate that the $n-$list $\lambda= (\lambda_1, \ldots , \lambda_n)$ is the spectrum of an  $n\times n$  nonnegative matrix $A$ with Perron eigenvalue $\lambda_1,$ we will write the first component with a semi-column as $\left(\lambda_1; \lambda_2,\ldots , \lambda_n\right)$ and say that $\lambda$  is realized by $A$ and that $A$ is a \emph{nonnegative realization} of $\lambda$.  In addition, the  $n-$list $\lambda= (\lambda_1;\lambda_2, \ldots , \lambda_n)$ is said to be  \emph{doubly stochastically realizable} if $\lambda$ is realized by a nonnegative element $A$ in $\Omega(n)$, and we shall say that $A$ is a \emph{doubly stochastic realization} of $\lambda$.

Recall that the inverse eigenvalue problem for special kind of matrices is concerned with constructing a matrix that maintains the required structure from its set of eigenvalues (see, e.g., \cite{Golub_05}).
The \emph{nonnegative inverse eigenvalue problem }(hereafter, the NIEP) can be stated as the problem of finding necessary and sufficient conditions for an $n-$tuples $\left(\lambda_1; \lambda_2, \ldots , \lambda_n\right)$ (where $\lambda_2, \ldots, \lambda_n$ might be complex) to be the spectrum of an $n\times n$ nonnegative matrix $A$ see \cite{Borobia_07,JLL_96,LL_78,Minc_88,Perfect_53,Perfect_55,Suleimanova_49} and the references therein. Similarly, the \emph{stochastic inverse eigenvalue problem }(SIEP) asks which sets of $n$ complex numbers can occur as the spectrum of an $n\times n$ stochastic matrix $A.$ In addition, the \emph{doubly stochastic inverse eigenvalue problem }denoted by (DIEP), is the problem of determining the necessary and sufficient conditions for a complex $n-$tuples to be the spectrum of an  $n\times n$  doubly stochastic matrix. Now the nonnegative $r-$generalized stochastic (resp. doubly stochastic) inverse eigenvalue problem can be defined analogously. However, for $r > 0$ it is obvious that this last problem is equivalent to that of (SIEP) (resp. DIEP) since $\left(r; \lambda_2, \ldots, \lambda_n\right)$ is realized by an $n\times n$ nonnegative $r-$generalized stochastic (resp. doubly stochastic) matrix if and only if $\dfrac{1}{r}\left(r; \lambda_2, \ldots, \lambda_n\right)$ is realized by an  $n\times n$  stochastic (resp. doubly stochastic) matrix.

It is well-known (see \cite{John_81}) that (NIEP) is equivalent to (SIEP). More precisely, if  $\left(\lambda_1; \lambda_2, \ldots, \lambda_n\right)$ is the spectrum of an  $n\times n$  nonnegative matrix $A,$ then $\left(\lambda_1; \lambda_2, \ldots, \lambda_n\right)$ is also the spectrum of a nonnegative $\lambda_1-$generalized stochastic matrix.
 On the other hand, (SIEP) and (DIEP) are known not to be equivalent (see \cite{John_81}) so the problem of characterizing which  stochastic matrices are similar (or cospectral) to doubly stochastic matrices is of interest as it sheds light on how these two problems differ.

In \cite{Guo_97}, Guo  proved that if $\left(\lambda_2,\ldots, \lambda_n\right)$ is any list of complex numbers which is closed under complex conjugation then there exists a least real number $\lambda_0$
with $\max\limits_{2\leq j\leq n}|\lambda_j|\leq\lambda_0\leq 2n\max\limits_{2\leq j\leq n}|\lambda_j|$
such that the list of complex numbers $\left( \lambda_1;\lambda_2,\ldots,\lambda_n\right)$ is realizable if and only if $\lambda_1\geq \lambda_0$. It is worthy to observe that by the Perron-Frobenius theorem, the lower bound is sharp. However, in \cite{soto_20},
 the authors  showed that the upper bound may be reduced to $(n-1)\max\limits_{2\leq j\leq n}|\lambda_j|$ in the case when at least one of the $\lambda_i$ is real. Furthermore,
  in \cite{enide,chris,maria} the exact value of $\lambda_0$ is found for  particular classes of matrices.

In this paper, we first describe a method for passing from a nonnegative realization to a doubly stochastic realization. As  applications, we give a new sufficient condition for a stochastic matrix $A$ to be cospectral to a doubly stochastic matrix $B$ and in this case $B$ is shown to be the unique closest doubly stochastic matrix to $A$ with respect to the Frobenius norm. Then, our next result deals with sharpening the upper bound for the nonnegative realization. A particular case of this last result  slightly improves the upper bound for the nonnegative realization  presented in \cite{soto_20}. For the case when none of the $\lambda_i$ is real little is known. So, our next next result is concerned with improving  Guo's upper bound for this case. Next, for doubly stochastic realization, we obtain an upper bound that improves Guo's bound as well. Finally, for certain particular cases, we give a further improvement of our last bound for doubly stochastic realization.

The main tool that we rely heavily on here  is  known as Brauer's theorem. Although, this theorem has been repeatedly used in the study of the NIEP, but the merit of using it here mainly lies in the applications where we are able to sharpen the upper bounds for the nonnegative and doubly stochastic  realizations.
Recently, in \cite{soto_20} the authors showed an important role for Brauer's theorem in the solvability of the (NIEP) and certainly our work here gives further support to this line of direction and  adds to this role another feature that can be manifested in the doubly stochastic realization problem. Brauer's theorem can be stated as follows.

\begin{theorem}
Let $A$ be any $n\times n$ matrix with eigenvalues $\lambda_1,\dots, \lambda_n.$ Let $v= (v_1,\ldots,v_n)^T$ be an eigenvectors of $A$ corresponding to the eigenvalue $\lambda_k$ and  let $z$ be any $n$-dimensional vector. Then the matrix $A + vz^T$ has eigenvalues
$\lambda_1,\dots, \lambda_{k-1}, \lambda_{k}+v^Tz,\lambda_{k+1},\ldots,\lambda_n$.
\end{theorem}

\section{Main Observations}
 We shall start with the following notation which is needed for our purposes and that will be used throughout this section. Let $A=\left(a_{ij}\right)$ be a \emph{real} $n\times n$ matrix, then for each $i=1,\ldots,n$, we  shall define the following quantity
$$w_j(A)= \left\{\begin{array}{c} \min\limits_{1\leq i\leq n} a_{ij}  \ \ \ \ \ \ \ \ \ \   \mbox{ if } \  a_{ij}< 0 \  \  \mbox{ for some } i=1,...,n,\ \ \  \ \\
\\
   0  \hspace{3.5 cm}\mbox{ otherwise.} \ \ \ \ \ \ \ \ \ \ \ \ \ 
\end{array}\right . $$ In addition, we shall define
by \begin{equation} w(A)= \min\{w_j(A), \ j=1,\ldots,n \}.\end{equation}

With these notation, we have  the following two elementary lemmas for which their proofs rely on Brauer's theorem.

\begin{lemma}
Let $\left(r, \lambda_2, \ldots, \lambda_n\right)$ be the spectrum of an $r-$generalized  $n\times n$  stochastic matrix $A = \left(a_{ij}\right).$ Then,  $\left(r + \sum\limits_{i=1}^{n}|w_i(A)|; \lambda_2, \ldots, \lambda_n\right)$ is the spectrum of an $n\times n$ nonnegative $\left(r + \sum\limits_{i=1}^{n}|w_i(A)|\right)$-generalized stochastic matrix $S.$
\end{lemma}
\begin{proof}
Since $e$ is an eigenvector for $A$ corresponding to the eigenvalue $r,$ then making use of Brauer's theorem with $z=\left(|w_1(A)|,\ldots,|w_n(A)|\right)^T$ gives
  an $n\times n$ generalized stochastic matrix $S = A+ez^T$ whose  spectrum is  $\sigma(S) = \left(r +e^Tz; \lambda_2,\ldots, \lambda_n\right),$ and obviously from the definition of the $w_i(A)$,  $S$ is nonnegative.
\end{proof}

\begin{lemma}
Let $\left(r, \lambda_2, \ldots, \lambda_n\right)$ be the spectrum of an $r-$generalized  $n\times n$  doubly stochastic matrix $A = \left(a_{ij}\right).$ Then, $\left(r + n|w(A)|; \lambda_2, \ldots, \lambda_n\right)$ is the spectrum of an $n\times n$ nonnegative $\left(r+n|w(A)| \right)$-generalized doubly stochastic matrix $S.$
\end{lemma}
\begin{proof}
As $e$ is an eigenvector for $A$ corresponding to the eigenvalue $r,$ then making use of Brauer's theorem this time  with $z=\left(|w(A)|,\ldots,|w(A)|\right)^T$ yields
  an $n\times n$ generalized doubly stochastic matrix $S = A+ez^T$ that has the desired spectrum  $\sigma(S) = \left(r +e^Tz; \lambda_2,\ldots, \lambda_n\right),$ and obviously from the definition of  $w(A)$, we conclude that $S$ is nonnegative.
\end{proof}

Now, for nonnegative matrices, we have the following.
\begin{theorem}\label{th:r+epsi}
Let $\sigma(A) = (r; \lambda_2, \ldots, \lambda_n)$ be the spectrum of an $n\times n$ nonnegative matrix $A = \left(a_{ij}\right).$ Then there exists a real $k_A \ge -r$ such that $(r + \varepsilon;\lambda_2, \ldots, \lambda_n)$ is the spectrum of a nonnegative $(r + \varepsilon)-$generalized $n \times n$ doubly stochastic matrix $D,$ for all $\varepsilon \ge k_A.$
\end{theorem}
\begin{proof}
 For simplicity, we let $s_j(A)$ be denoted by $s_j$ for all $j = 1, 2,\ldots, n$ and without loss of generality, we can assume that $A$ is nonnegative $r-$generalized stochastic matrix. Then for any vector $z = (z_1,\ldots,z_n)^T\in \mathbb{R}^n,$ Brauer's theorem tells us that the spectrum $\sigma(B)$ of the matrix $B=A+e z^T$ is clearly equal to
\begin{equation*}
\sigma(B)=\left( r+\sum_{j=1}^{n}z_j;\lambda_2,\ldots,\lambda_n\right).
\end{equation*}
Now the basic idea is to study the conditions for which $B$ is a nonnegative generalized doubly stochastic matrix. Clearly, the matrix $B$ is given by:
\begin{equation*}
B=\begin{pmatrix}
a_{11}&\cdots&a_{1n-1}&r-\sum_{j=1}^{n-1}a_{1j}\\
\vdots&&\vdots&\vdots\\
a_{n1}&\cdots&a_{nn-1}&r-\sum_{j=1}^{n-1}a_{nj}
\end{pmatrix} +\begin{pmatrix}
	z_1&\cdots&z_n\\
	\vdots&&\vdots\\
	z_1&\cdots&z_n
\end{pmatrix}.
\end{equation*}
Also observe that each row sum of $B$ is equal to $r+\sum_{j=1}^{n}z_j,$ and for all $j = 1, 2,\ldots,n,$ the sum of the $j$th column of $B$ is $s_j + nz_j.$ By equating each $j$th column sum of $B$ to the $j$th row sum which is $r+\sum_{j=1}^{n}z_j,$ we obtain the system of $n$ linear equations in the $n$ unknowns $z_1,\ldots, z_n$ given by:
\begin{equation}\label{syst:1}
\left\{\begin{aligned}
nz_1+s_1&=r+\sum_{j=1}^{n}z_j\\
  &\ \ \! \vdots \\
nz_j+s_j&=r+\sum_{j=1}^{n}z_j\\
  &\ \ \! \vdots \\
nz_n+s_n&=r+\sum_{j=1}^{n}z_j.
\end{aligned}
\right.
\end{equation}
Since $\sum_{j=1}^{n}s_j=nr,$ then the sum of any $(n - 1)$ equations of $\left(\ref{syst:1}\right)$ is equal to the remaining equation and the system is in fact of $(n -1)$ equations in $n$ unknowns and has an infinite number of solutions (which is obviously a line solution). Now, if we let $s_m = \max\left(s_j\right),$  and we let $z_m$ be the only parameter for the system and then clearly $\sum_{j=1}^{n}z_j=nz_m+s_m-r$.  Thus,  its solution is easily given by:
\begin{equation}\label{syst:2}
\left\{\begin{aligned}
	z_1&=z_m+\dfrac{s_m-s_1}{n}\\
z_2&=z_m+\dfrac{s_m-s_2}{n}\\
	&\ \ \!\vdots\\
z_n&=z_m+\dfrac{s_m-s_n}{n}.
\end{aligned}
\right.
\end{equation}
So that the $(i,j)-$entry of $B =\left(b_{ij}\right)$ is clearly given by:
$$
b_{ij}=a_{ij}+ z_m+\dfrac{s_m-s_j}{n}
$$
As $a_j$ is the smallest entry of the $j$th column of $A,$ then obviously the conditions for which $B$ is nonnegative are given by:
\begin{equation}\label{Isyst:3}
\left\{\begin{aligned}
a_1+z_m+\dfrac{s_m-s_1}{n}&\ge 0\\
&\ \ \!\vdots\\
a_{m-1}+z_m+\dfrac{s_m-s_{m-1}}{n}&\ge 0\\
a_m+z_m&\ge 0\\
a_{m+1}+z_m+\dfrac{s_m-s_{m+1}}{n}&\ge 0\\
&\ \ \!\vdots\\
a_n+z_m+\dfrac{s_m-s_n}{n}&\ge 0.
\end{aligned}
\right.
\end{equation}
It is easy to see that system $\left(\ref{Isyst:3}\right)$ has an infinite number of solutions, and if $c= \min\limits_{1\leq j\leq n}\{a_j+\dfrac{s_m-s_j}{n}\}$  then
$z_m=-c$ \ is the smallest value for which $\left(\ref{Isyst:3}\right)$ is satisfied. Now as $s_m\geq r$, then we can write  $s_m-nc=r+k_A$ where
$k_A=s_m-nc-r\le r+k_A$. Also,  since $s_m-nc\geq 0$, then clearly
   $k_A\ge -r.$ Now for any $\varepsilon\ge k_A$ there exists $\alpha\ge 0$ such that $\varepsilon = k_A + \alpha$ and then $(r + \varepsilon; \lambda_2,\ldots, \lambda_n)$ is obviously realized by $D=B + \alpha J_n$ and the proof is complete.
\end{proof}

\begin{example}
Let $A$ be the stochastic matrix defined by $$A=\begin{pmatrix}
	\frac{1}{3}&\frac{1}{3}&\frac{1}{3}\\[10pt]
	\frac{1}{4}&\frac{1}{4}&\frac{1}{2}\\[10pt]
	\frac{1}{6}&\frac{1}{6}&\frac{2}{3}
\end{pmatrix}$$ with $\sigma(A)=\left(1;0,\frac{1}{4}\right).$
Then $s_1=s_2=\frac{3}{4}, s_3=\frac{3}{2}$ and therefore
according to the proof of the preceding theorem, $s_m=s_3$ and
\begin{equation*}
\begin{aligned}
C&=\begin{pmatrix}
	\frac{1}{3}&\frac{1}{3}&\frac{1}{3}\\[10pt]
	\frac{1}{4}&\frac{1}{4}&\frac{1}{2}\\[10pt]
	\frac{1}{6}&\frac{1}{6}&\frac{2}{3}
\end{pmatrix}+\begin{pmatrix}
\frac{1}{4}& \frac{1}{4}&0\\[10pt]
 \frac{1}{4}&\frac{1}{4}&0\\[10pt]
\frac{1}{4}&\frac{1}{4}&0
\end{pmatrix}\\[16pt]
&=\begin{pmatrix}
\frac{7}{12}&\frac{7}{12}&\frac{1}{3}\\[10pt]
\frac{1}{2}&\frac{1}{2}&\frac{1}{2}\\[10pt]
\frac{5}{12}&\frac{5}{12}&\frac{2}{3}
\end{pmatrix}.
\end{aligned}
\end{equation*}
Clearly the smallest entry of $C$ is $c= \dfrac{1}{3}$ and therefore we obtain the optimal doubly stochastic realization for $\sigma(A)$ given by
 $$B =\begin{pmatrix}
\dfrac{1}{4}&\dfrac{1}{4}&0\\[10pt]
\dfrac{1}{6}&\dfrac{1}{6}&\dfrac{1}{6}\\[10pt]
\dfrac{1}{12}&\dfrac{1}{12}&\dfrac{1}{3}
\end{pmatrix}$$ whose spectrum is $\sigma(B) = \left(\dfrac{1}{2}; 0, \dfrac{1}{4}\right).$
Finally, note that $B+\dfrac{1}{2}J_n$ is doubly stochastic with spectrum $(1; 0, \frac{1}{4}).$
\end{example}

\begin{remark}
\begin{itemize} 
\item It should be stressed that the boundary $k_A=-r$ can be achieved by the matrix $A$ whose each entry in the first column is $r$ and all the remaining entries are zeroes and in this case, $A$ is cospectral to the nonnegative $r-$generalized doubly stochastic matrix $rJ_n.$
   \item  The proof of the preceding theorem gives rise to an algorithm that describes how to go from nonnegative realization to ` doubly stochastic realization.
\item One of the obvious applications of the preceding theorem lies in the fact that any known sufficient conditions for the resolution of (NIEP) will eventually lead to some sufficient conditions for the resolution of (DIEP).

\end{itemize}
\end{remark}

\subsection{Applications}
 In order to present another application, we shall start this section by recalling some auxiliary results from \cite{Khoury_98} for which we need to introduce some more relevant notation. Let $V_{n-1} = I_{n-1}- \left(1+\dfrac{1}{\sqrt{n}}\right)J_{n-1}$ and define the block matrix $U_n =\begin{pmatrix}
\dfrac{1}{\sqrt{n}}&\dfrac{\sqrt{n-1}}{\sqrt{n}}e^T_{n-1}\\[16pt]
\dfrac{\sqrt{n-1}}{\sqrt{n}}e_{n-1}&V_{n-1}
\end{pmatrix}$ where $ e_n= \frac{1}{\sqrt{n}} e.$  Then the first result can be stated as follows.
\begin{lemma}\label{lem:Decomp_D.S}
For any matrix $A \in\Omega^1(n),$ there exists an $(n-1)\times(n-1)$ matrix $X$ such that $A = U_n\left(1\oplus X\right)U_n$ and conversely, for any $(n-1) \times (n-1)$ real matrix $X,$ it holds that $U_n\left(1\oplus X\right)U_n \in\Omega^1(n).$
\end{lemma}
\begin{remark}
The preceding lemma is also valid if $U_n$ is replaced by any real orthogonal matrix $V$ whose first column is $e_n$ and in this case $A = V \left(1 \oplus X\right)V^T$ (see {\cite[Theorem~1]{Sinkhorn_81}}).
\end{remark}
The second one is the following theorem.
\begin{theorem}[see~{\cite[p. 564]{Khoury_98}}]\label{th:Nearest_D.S}
Let $A$ be an  $n\times n$  real matrix. Then
\begin{equation*}
B^* = \left(I_n - J_n\right)A\left(I_n - J_n\right) + J_n
\end{equation*}
is the unique closest matrix to $A$ in $\Omega^1(n)$ with respect to the Frobenius norm.
\end{theorem}

Recall that in \cite{John_81}, the author obtained some sufficient conditions for a stochastic matrix to be similar to a doubly stochastic matrix. In connection with this, our next result is concerned with
 obtaining sufficient conditions for a stochastic matrix to be cospectral to a doubly stochastic matrix.

\begin{theorem}
Let $A = \left(a_{ij}\right)$ be an  $n\times n$  stochastic matrix with spectrum $(1; \lambda_2,\ldots, \lambda_n).$ Let $a_j$ and $s_j$ be the smallest entry and the sum of the $j$th column of $A$ respectively. If
\begin{equation}
\left\{\begin{aligned}\label{Isyst:4}
s_1&\le 1+na_1\\
s_2&\le 1+na_2\\
&\ \ \! \vdots\\
s_n&\le 1+na_n
\end{aligned}\right.
\end{equation}
then there exists an  $n\times n$  doubly stochastic matrix $B$ with spectrum $\sigma(B)=(1;\lambda_2,\ldots, \lambda_n).$ Moreover, $B$ is the unique closest doubly stochastic matrix to $A$ with respect to the Frobenius norm and with the property that $B$ is cospectral to $A.$
\end{theorem}
\begin{proof}
 As $A$ is stochastic then clearly  $s_1 +s_2 +\cdots+s_n=n.$  Now observing that once (5) is valid then
$z_m =\dfrac{1}{n}-\dfrac{s_m}{n}$ is one of the solutions of  system $\left(\ref{Isyst:3}\right).$ Also, it is not hard to check that in this case, the matrix $B$ in the proof of Theorem~\ref{th:r+epsi} is doubly stochastic and has the required spectrum. For the second part, Theorem~\ref{th:Nearest_D.S} tells us that it suffices to prove that $B=\left(I_n - J_n\right)A\left(I_n - J_n\right) + J_n.$ Since for a stochastic matrix $A,$ we have $AJ_n = J_n$ then $\left(I_n-J_n\right)A\left(I_n-J_n\right)+J_n=A-J_nA+J_n.$ Moreover, each $(i,j)-$entry of the matrix $J_nA$ is clearly equals to $\dfrac{s_j(A)}{n}.$ With this in mind, a simple check now shows that all entries of the two matrices $B$ and $A - J_nA + J_n$ are the same and the proof is complete.
\end{proof}
\begin{example}
Consider the following stochastic matrix $$A=\begin{pmatrix}
\dfrac{2}{3}&\dfrac{1}{3}&0\\[10pt]
\dfrac{1}{3}&\dfrac{2}{3}&0\\[10pt]
\dfrac{1}{2}&\dfrac{1}{2}&0
\end{pmatrix}$$ which has spectrum $\sigma(A)=\left(1,\dfrac{1}{3},0\right).$ Clearly $A$ satisfies the conditions of the preceding theorem and from the above argument, $A$ is cospectral to the matrix $$B=\begin{pmatrix}
\dfrac{1}{2}&\dfrac{1}{6}&\dfrac{1}{3}\\[10pt]
\dfrac{1}{6}&\dfrac{1}{2}&\dfrac{1}{3}\\[10pt]
\dfrac{1}{3}&\dfrac{1}{3}&\dfrac{1}{3}
\end{pmatrix}$$ which is the unique doubly stochastic matrix that is closest to $A$ with respect to the Frobenius norm.
\end{example}
More general results are given in the following.
\begin{theorem}
Let $A = \left(a_{ij}\right)$ be an $n\times n$ nonnegative $r-$generalized stochastic matrix with spectrum $\left(r; \lambda_2,\ldots, \lambda_n\right)$ and  smallest entry $a$ and whose $i$-th column sum is denoted by $s_i$. In addition, let $a_i$ be the smallest entry of the $i$-th column and $s_m$ be the maximum column sum of $A$.  Then, the followings are true.
\begin{enumerate}
\item  There exists an  $n\times n$  nonnegative generalized doubly stochastic matrix $B$ with spectrum
$$\sigma(B)=\left(nr-na; \lambda_2,\ldots, \lambda_n\right).$$
\item   If $s_i\leq r+na_i$ for all $i=1,\cdots,n,$ then there exists an  $n\times n$  nonnegative generalized doubly stochastic matrix $C$ with spectrum $\sigma(C)=\left(r; \lambda_2,\ldots, \lambda_n\right).$
\item  If $s_m-s_i\geq n(a_m-a_i)$ for all $i=1,\cdots,n,$ then there exists an  $n\times n$  nonnegative generalized doubly stochastic matrix $D$ with spectrum $\sigma(D)=\left(s_m-na_m; \lambda_2,\ldots, \lambda_n\right).$
\end{enumerate}
\end{theorem}
\begin{proof}
\begin{enumerate}
\item In System $(\ref{Isyst:3}),$ let $z_m =r-\dfrac{s_m}{n}-a.$ Since each entry of $A$ is less than or equal $r$ then $s_m \le nr$. Therefore, with this choice of $z_m,$  System $\left(\ref{Isyst:3}\right)$ is not violated and as a conclusion, we obtain a nonnegative generalized doubly stochastic matrix $B$ whose each row and column sum equals to $nr-na,$ and whose spectrum is $\sigma(B)=\left(nr-na; \lambda_2,\ldots, \lambda_n\right).$
	\item Choosing $z_m =\dfrac{r}{n}-\dfrac{s_m}{n},$ and noticing that the hypothesis taken in this case insures that System $\left(\ref{Isyst:3}\right)$ is satisfied. Thus,  there exists an  $n\times n$  nonnegative generalized doubly stochastic matrix $C$ with spectrum $\sigma(C)=\left(r; \lambda_2,\ldots, \lambda_n\right).$
\item
It suffices to take  $z_m =-a_m,$  and then note that with the hypothesis assumed in this case,  it is easy to check that  System $\left(\ref{Isyst:3}\right)$ is satisfied and therefore, we obtain a nonnegative generalized doubly stochastic matrix $D$ whose each row and column sum equals to $\left(s_m-na_m\right),$ and whose spectrum is $\sigma(D)=\left(s_m-na_m; \lambda_2,\ldots, \lambda_n\right).$
\end{enumerate} \end{proof}
\section{Improving the upper bound for nonnegative realizations}
In this section, we will start with an $(n-1)$-list $\left(\lambda_2,\ldots, \lambda_n\right)$  of complex numbers which is closed under complex conjugation and study the problem of finding the optimal $\lambda_1$ for a particular method described here such that $\left(\lambda_1;\lambda_2,\ldots, \lambda_n\right)$ has an $n\times n$ nonnegative realization.

Recall that in \cite{Guo_97}, Guo proved the following theorem which deals with nonnegative realization.
  \begin{theorem}[{\cite[Theorem~2.1]{Guo_97}}] If $\left(\lambda_2,\ldots, \lambda_n\right)$ is any list of complex numbers which is closed under complex conjugation then there exists a least real number $\lambda_0$
with $\max\limits_{2\leq j\leq n}|\lambda_j|\leq\lambda_0\leq 2n\max\limits_{2\leq j\leq n}|\lambda_j|$
such that the list of complex numbers $\left( \lambda_1;\lambda_2,\ldots,\lambda_n\right)$ is realizable if and only if $\lambda_1\geq \lambda_0$.
\end{theorem}
 In \cite{soto_20} the authors showed that the upper bound in the preceding theorem may be reduced to $(n-1)\max\limits_{2\leq j\leq n}|\lambda_j|$
  in the case when at least one of them is real. Indeed, the authors proved that  this upper bound  is sharp by exhibiting the $(n-1)$-list $(-1,\ldots,-1)$. Despite this, we shall present a sharper bound which shows that an improvement is possible in a certain direction. Before doing this,  we shall first recall that
their proof  relies on  setting $m=\max\limits_{2\leq j\leq n}|\lambda_j|$ and  defining $\mu_i=\frac{\lambda_i}{m(n-1)}$, $i=2,\ldots,n$. Now, if
$\mu_2,\ldots,\mu_p$ are real and if $x_i=\Re(\mu_i)$,  $y_i=\Im(\mu_i)$ for $i=p+1,\ldots,\frac{n+p}{2}$, then
  they consider the $n\times n$ matrix $B$ (whose each row sum is zero)  defined by
\begin{equation}
		B=\begin{pmatrix}
			0& 0 & 0 & \ldots & \ldots &\ldots &\ldots &\ldots & 0\\
			-\mu_2&\mu_2 &\ddots & & & & & & \vdots\\
            \vdots& \ddots & \ddots & \ddots &  &  &  &  & \vdots\\
			-\mu_p& 0 & \cdots & \mu_p & 0 & & & & \vdots\\
			-x_s& y_s & \ddots & 0 &x_s &-y_s & & & \vdots\\
           -x_s& -y_s &0 & \ddots &y_s &x_s &\ddots & & \vdots\\
			\vdots&\vdots &\vdots &\ddots &0 & \ddots& \ddots&\ddots& 0\\
			-x_t& y_t & \vdots & \vdots &  & \ddots & \ddots & x_t & -y_t\\
			-x_t& -y_t & 0 & \ldots &\ldots &\ddots &0 & y_t & x_t\\
		\end{pmatrix},
	\end{equation}
and  make use of Brauer's theorem with a convenient chosen $n$-vector $z$ with one component equals to zero and  the remaining $n-1$ entries are equal to $\frac{1}{n-1}$ to obtain an $n\times n$ nonnegative matrix $A=B+ez^T$ with spectrum $((n-1)m;\lambda_2,\ldots,\lambda_n)$ and whose each row sum is $(n-1)m$. Their result can be summarized in the following.
 \begin{theorem}[{\cite[Theorem~3.3]{soto_20}}] If $\left(\lambda_2,\ldots, \lambda_n\right)$ is any list of complex numbers which is closed under complex conjugation with $\lambda_2\geq \ldots\geq \lambda_p   $  where $2\leq p\leq n$. Then there exists a least real number $\lambda_0$
with $\max\limits_{2\leq j\leq n}|\lambda_j|\leq\lambda_0\leq (n-1)\max\limits_{2\leq j\leq n}|\lambda_j|$
such that the list of complex numbers $( \lambda_1;\lambda_2,\ldots,\lambda_n)$ is realizable if and only if $\lambda_1\geq \lambda_0$.
\end{theorem}
\subsection{First case when  at least one of $\{ \lambda_2,...,\lambda_n\}$ is real }
Here we present the most general results of this kind using Brauer's theorem by starting
  a different matrix $C$ and with a more convenient vector $z$ in the case when at least one of the eigenvalues is real. A particular case of our results here improves the upper bound  given in the preceding theorem.

   Indeed, suppose that $\lambda_2,\ldots,\lambda_p$ are real  and $\lambda_{p+1},\ldots,\lambda_n$ are non-real with their real parts and imaginary parts are respectively denoted by $x_i=\Re(\lambda_i)$ and  $y_i=\Im(\lambda_i)$ for $i=p+1,\ldots,\frac{n+p}{2}$, and let $C$ be the $n\times n$ matrix given by
 {\tiny   \begin{equation}\label{Eq::C}
		C=\begin{pmatrix}
			0& 0 & 0 & \ldots & \ldots &\ldots &\ldots &\ldots & 0\\
			-\lambda_2&\lambda_2 &\ddots & & & & & & \vdots\\
            \vdots& \ddots & \ddots & \ddots &  &  &  &  & \vdots\\
			-\lambda_p& 0 & \cdots & \lambda_p & 0 & & & & \vdots\\
			-x_{p+1}& y_{p+1} & \ddots & 0 &x_{p+1} &-y_{p+1} & & & \vdots\\
           -x_{p+1}& -y_{p+1} &0 & \ddots &y_{p+1} &x_{p+1} &\ddots & & \vdots\\
			\vdots&\vdots &\vdots &\ddots &0 & \ddots& \ddots&\ddots& 0\\
			-x_\frac{n+p}{2}& y_\frac{n+p}{2} & \vdots & \vdots &  & \ddots & \ddots & x_\frac{n+p}{2} & -y_\frac{n+p}{2}\\
			-x_\frac{n+p}{2}& -y_\frac{n+p}{2} & 0 & \ldots &\ldots &\ddots &0 & y_\frac{n+p}{2} & x_\frac{n+p}{2}\\
		\end{pmatrix},
	\end{equation}}
Now, consider the following 3 cases:\\
\underline{Case 1:} If $\Re\left(\lambda_i\right)\leq 0$ for all $i= 2, 3,\ldots, n,$  then all the entries in the first column of $C$ are nonnegative, therefore
by using Brauer's theorem with $z=(0,|w_2(C)|,\ldots,|w_n(C)|)^T$, we obtain a nonnegative matrix $A=C+ez^T$ with spectrum  $ (|w_2(C)|+\cdots+|w_n(C)|; \lambda_2,\ldots,\lambda_n)$ and such that each row sum of $A$ is $|w_2(C)|+\cdots+|w_n(C)|$.\\
\underline{Case 2:} If $\Re\left(\lambda_k\right)\ge 0$ for some $k\in S_1\bigcup S_2$ where $S_1=\{ 3,\ldots,p\}$ and $S_2= \{ p+1,\ldots,n\}$  then all the entries in the $k$-th column of $C$ are nonnegative if $k$ is in $S_1$ and all entries in either the $k$-th  or  the $(k +1)$-th column  of $C$ are nonnegative if $k$ is in $S_2$. So that making use of  Brauer's theorem with
 $z_1=(|w_1(C)|,\ldots, |w_{k-1}(C)|, 0, |w_{k+1}(C)|,\ldots,|w_n(C)|)^T$ if the $k$-th  column  of $C$ is nonnegative to obtain a nonnegative matrix $A=C+ez_1^T$ with spectrum  $ (|w_1(C)|+\cdots+|w_{k-1}(C)|+ |w_{k+1}(C)|+\cdots+|w_n(C)|; \lambda_2,\ldots,\lambda_n)$ and such that each row sum of $A$ is $|w_1(C)|+\cdots+|w_{k-1}(C)|+ |w_{k+1}(C)|+\cdots+|w_n(C)|$  and applying it with $z_2=(|w_1(C)|,\ldots, |w_{k}(C)|, 0, |w_{k+2}(C)|,\ldots,|w_n(C)|)^T$ if  the $(k +1)$-th column  of $C$ is nonnegative,
   to obtain a nonnegative matrix $A=C+ez_2^T$ with spectrum  $ (|w_1(C)|+\cdots+|w_{k}(C)|+ |w_{k+2}(C)|+\cdots+|w_n(C)|; \lambda_2,\ldots,\lambda_n)$ and such that each row sum of $A$ is $|w_1(C)|+\cdots+|w_{k}(C)|+ |w_{k+2}(C)|+\cdots+|w_n(C)|$.\\
\underline{Case 3:} If $\Re\left(\lambda_2\ge 0\right)$ and $\Re\left(\lambda_i\right)\leq 0$ for all $i= 3,\ldots, n,$   then we use the matrix $C'$ which is obtained from $C$ by exchanging all entries in the first and the second columns except the second entries in those two columns remain unchanged i.e. in $C'$  we leave  $-\lambda_2$ in the first column and $\lambda_2$  in the second column without changing. So that all the entries in the second column of $C'$ are nonnegative, and noticing that $w_1(C')=w_2(C)$, $w_i(C')=w_i(C)$ for all $i=3,\ldots,n$, then
by using Brauer's theorem with $z=(|w_2(C)|, 0, |w_3(C)|,\ldots,|w_n(C)|)^T$, we obtain a nonnegative matrix $A=C'+ez^T$ with spectrum  $ (|w_2(C)|+\cdots+|w_n(C)|; \lambda_2,\ldots,\lambda_n)$ whose each row sum $|w_2(C)|+\cdots+|w_n(C)|$.

Now a simple check shows that the quantity $w(C)$ which is defined by (2.1) is given by
 $$w(C)=\min\{\pm \lambda_2,\ldots,\pm \lambda_p,\pm x_{p+1}, \pm y_{p+1} ,\ldots, \pm x_\frac{n+p}{2}, \pm y_\frac{n+p}{2} \}.$$   With this in mind,  we  conclude the following particular case  which gives a refinement of the previous theorem.
\begin{theorem} If $\left(\lambda_2,\ldots, \lambda_n\right)$ is any list of complex numbers which is closed under complex conjugation
 with $\lambda_2\geq \ldots\geq \lambda_p   $  where $2\leq p\leq n$ and let $\lambda_{p+1},\ldots,\lambda_n$ be non-real with their real parts and imaginary parts are respectively denoted by $x_i=\Re(\lambda_i)$ and  $y_i=\Im(\lambda_i)$ for $i=p+1,\ldots,\frac{n+p}{2}$. Moreover, let $$w= \min\{\pm \lambda_2,\ldots,\pm \lambda_p,\pm x_{p+1}, \pm y_{p+1} ,\ldots, \pm x_\frac{n+p}{2}, \pm y_\frac{n+p}{2} \}.$$  Then, there exists a least real number $\lambda_0$
with $\max\limits_{2\leq j\leq n}|\lambda_j|\leq\lambda_0\leq (n-1)|w|$
such that the list of complex numbers $( \lambda_1;\lambda_2,\ldots,\lambda_n)$ is realizable if and only if $\lambda_1\geq \lambda_0$.
\end{theorem}

\subsection{Second case when  none of $\{ \lambda_2,...,\lambda_n\}$ is real }
 Here we deal with the case when all the entries in the list $\left(\lambda_2,\ldots, \lambda_n\right)$ are non-real. It is worthy to mention that this case has not been addressed before in the literature. Indeed, in this case $n$ has to be odd with $n=2s-1$ and we use the same analysis as above but with the $n\times n$ matrix $E$ given by
 \begin{equation}\label{Eq::E}
		E=\begin{pmatrix}
			0& 0 & 0 & \ldots & \ldots &\ldots &\ldots  & 0\\
			-x_2+y_2&x_2 & -y_2 &\ddots & & & &  \vdots\\
            -x_2-y_2&y_2 & x_2 & 0 &\ddots & & &  \vdots\\
			-x_{3}& y_{3} & 0  &x_{3} &-y_{3} & \ddots & &   \vdots\\
            -x_{3}& - y_{3} & 0 &y_{3} & x_{3} & 0 &\ddots &  \vdots\\
			\vdots&\vdots &\vdots  & 0 &0 & \ddots& \ddots& 0 \\
			-x_{s}& y_{s} & 0 & \vdots & \ddots  & \ddots &  x_s & -y_{s}\\
			-x_{s}& -y_{s} & 0 & \ldots &\ldots & 0 & y_{s} & x_{s}\\
		\end{pmatrix}
	\end{equation}
where $x_i=\Re(\lambda_i)$ and  $y_i=\Im(\lambda_i)$ for $i=2,\ldots,s$.
We argue as before except in this case we need only to  consider two cases with the matrix $E$ and some variant of it:\\
\underline{Case 1:} If $\Re\left(\lambda_i\right)\leq 0$ for all $i= 2,\ldots, s,$  then in the first column  either $y_2$ or $-y_2$ is non-positive. Without loss of generality, suppose that $y_2\leq 0$ then we first apply Brauer's theorem on $E$ with  $z=(-y_2,0,y_2,0,\ldots,0)^T$ to obtain  the matrix $E'$ given by
\begin{equation}\label{Eq::E'}
		E'=\begin{pmatrix}
			-y_2& 0 & y_2& \ldots & \ldots &\ldots &\ldots  & 0\\
			-x_2&x_2 & 0 &\ddots & & & &  \vdots\\
            -x_2-2y_2& y_2 & x_2+ y_2  &\ddots & & &  \vdots\\
			-x_{3}-y_2& y_{3} & y_2 &x_{3} &-y_{3} & \ddots & &   \vdots\\
            -x_{3}-y_2& - y_{3} & y_2 &y_{3} & x_{3} & 0 &\ddots &  \vdots\\
			\vdots&\vdots &\vdots  & 0 &0 & \ddots& \ddots& 0 \\
			-x_{s}-y_2& y_{s} & y_2 & \vdots & \ddots  & \ddots &  x_s & -y_{s}\\
			-x_{s}-y_2& -y_{s} & y_2 & \ldots &\ldots & 0 & y_{s} & x_{s}\\
		\end{pmatrix}
	\end{equation}
 whose each row sum is zero  and whose first column is nonnegative, therefore
by using Brauer's theorem with $z=(0,|w_2(E')|,\ldots,|w_n(E')|)^T$, we obtain a nonnegative matrix $A=E'+ez^T$ whose spectrum is given by $ \left (|w_2(E')|+\cdots+|w_n(E')|; \lambda_2,\ldots,\lambda_n\right)$ and such that each row sum of $A$ is $|w_2(E')|+\cdots+|w_n(E')|$.\\
\underline{Case 2:} If $x_k=\Re\left(\lambda_k\right)\ge 0$ for some $k=2,\ldots,s$   then all entries in either the $k$-th  or  the $(k +1)$-th column  of $E$ are nonnegative. Applying Brauer's theorem with
 $z_1=(|w_1(E)|,\ldots, |w_{k-1}(E)|, 0, |w_{k+1}(E)|,\ldots,|w_n(E)|)^T$ if the $k$-th  column  of $E$ is nonnegative
  to obtain a nonnegative matrix $A=E+ez_1^T$ whose  spectrum is given by  $$ \sigma(A)=\left(|w_1(E)|+\cdots+|w_{k-1}(E)|+ |w_{k+1}(E)|+\cdots+|w_n(E)|; \lambda_2,\ldots,\lambda_n\right)$$ and whose each row sum is $|w_1(E)|+\cdots+|w_{k-1}(E)|+ |w_{k+1}(E)|+\cdots+|w_n(E)|$  and applying it with $z_2=(|w_1(E)|,\ldots, |w_{k}(E)|, 0, |w_{k+2}(E)|,\ldots,|w_n(E)|)^T$ if  the $(k +1)$-th column  of $E$ is nonnegative,
   to obtain a nonnegative matrix $A=E+ez_2^T$ with spectrum given by  $$ \sigma(A)=\left (|w_1(E)|+\cdots+|w_{k}(E)|+ |w_{k+2}(E)|+\cdots+|w_n(E)|; \lambda_2,\ldots,\lambda_n\right )$$ and whose each row sum is $|w_1(E)|+\cdots+|w_{k}(E)|+ |w_{k+2}(E)|+\cdots+|w_n(E)|.$

Next, we study a particular case in the preceding discussion. In fact,  if $w= \min\{\pm x_{2}, \pm y_{2} ,\ldots, \pm x_{s}, \pm y_{s} \},$ then observe that in order to  obtain from $E'$ a nonnegative generalized stochastic matrix  via the use of Brauer's theorem, we need to add $2|w|$ to each entry of the third column while it is enough to add $|w|$ to each entry of the $i$-th column of $E'$ for $i=2,4,5,\ldots,n$ (as the first column is nonnegative). Similarly, for the matrix $E$, we need to add to each entry of the first column $2|w|$ and for another $n-2$ columns we only need to add to each entry of those $|w|$ to insure that we obtain a nonnegative  generalized stochastic matrix via the use of Brauer's theorem. Thus, we have the following.

\begin{theorem} If $\left(\lambda_2,\ldots, \lambda_n\right)$ with $n=2s-1$, is any list of complex numbers which is closed under complex conjugation
  and such that none of them is real and where their real parts and imaginary parts are respectively denoted by $x_i=\Re(\lambda_i)$ and  $y_i=\Im(\lambda_i)$ for all $i=2,\ldots,s$. Moreover, let $w= \min\{\pm x_{2}, \pm y_{2} ,\ldots, \pm x_{s}, \pm y_{s} \},$  Then, there exists a least real number $\lambda_0$
with $\max\limits_{2\leq j\leq n}|\lambda_j|\leq\lambda_0\leq n|w|$
such that the list of complex numbers $( \lambda_1;\lambda_2,\ldots,\lambda_n)$ is realizable if and only if $\lambda_1\geq \lambda_0$.
\end{theorem}

\section{Improving the upper bound for doubly stochastic realizations}
 Our main objective in this section, is to study the same problem for doubly stochastic realizations. Of course, one expects weaker results as more constraints are imposed here. In fact, we have the following theorem.
 \vspace{0.5cm}
 \begin{theorem} If $\left(\lambda_2,\ldots, \lambda_n\right)$ is any list of complex numbers which is closed under complex conjugation and let
  $$w= \min\{\pm \Re\left(\lambda_2\right),\ldots,\pm \Re\left(\lambda_n\right) \}.$$  Then, there exists a least real number $\lambda_0$
with $\max\limits_{2\leq j\leq n}|\lambda_j|\leq\lambda_0\leq (n+2)|w|$
such that the list of complex numbers $( \lambda_1;\lambda_2,\ldots,\lambda_n)$ is realizable by an $n\times n$ nonnegative generalized doubly stochastic matrix  if and only if $\lambda_1\geq \lambda_0$.
\end{theorem}

 \begin{proof} We split the proof into two cases.\\
 \underline{Case 1:} \\
 Suppose that $\lambda_2,\ldots,\lambda_p$ are real  and $\lambda_{p+1},\ldots,\lambda_n$ are non-real with their real parts and imaginary parts are respectively denoted by $x_i=\Re(\lambda_i)$ and  $y_i=\Im(\lambda_i)$ for $i=p+1,\ldots,\frac{n+p}{2}$. Moreover,
  denote the $m\times q$ zero matrix by      $\bf{0_{m\times q}}$. Also, we define the following matrices:

\begin{equation}
		\Lambda_{p}= \begin{pmatrix}
			0& 0&  \ldots & \ldots & 0 \\
			-\lambda_2&\lambda_2& 0 &\ldots & 0\\
            0 &-\lambda_3 & \lambda_3  & \ddots & \vdots\\
            \vdots & \ddots & \ddots & \ddots &  0 \\
			0 & \ldots & 0 &  -\lambda_p & \lambda_p \\
			
		\end{pmatrix},
	\end{equation} which is a $p\times p$ matrix,
{\tiny \begin{equation}
		Z= \begin{pmatrix}
			x_{p+1}& -y_{p+1} & 0 & 0 & 0 & 0 &\ldots & & & 0\\
           y_{p+1}& x_{p+1} &0 & 0 &\ddots &0 &\ldots & & & \vdots\\
			-x_{p+2}& y_{p+2} & x_{p+2} &-y_{p+2} &\ddots & \ddots & \ddots & & & \vdots\\
           -y_{p+2}& -x_{p+2} & y_{p+2} &x_{p+2} & 0 & 0 &\ddots &\ddots & & \vdots\\
           0& 0 & -x_{p+3} & y_{p+3} & x_{p+3} &-y_{p+3} & 0 & \ddots\\
           0& 0 & -y_{p+3} &-x_{p+3} &  y_{p+3} &x_{p+3} &0 &\ddots & & \vdots\\
			0&0  &0 & 0 &\ddots & \ddots& \ddots&\ddots& 0&0\\
            \vdots&\vdots &\vdots &\vdots &\ddots & \ddots& \ddots&\ddots& 0 &0\\
			0 & 0  & \vdots & \vdots & & & -x_\frac{n+p}{2} &  y_\frac{n+p}{2}\  & x_\frac{n+p}{2} & -y_\frac{n+p}{2}\\
			0& 0 & \ldots & \ldots &\ldots & 0 &-y_\frac{n+p}{2} & -x_\frac{n+p}{2} &  y_\frac{n+p}{2} & x_\frac{n+p}{2}\\
		\end{pmatrix},
	\end{equation}}  which is a $(n-p)\times (n-p)$ matrix, and
 the $(n-p)\times p $ matrix given by \begin{equation}
		L= \begin{pmatrix}
			-x_{p+1}& y_{p+1} & 0  & \ldots & 0 \\
           -y_{p+1}& -x_{p+1} &0  &\ddots & \vdots\\
			0&0  &0  &\ddots & \vdots\\
            \vdots&\vdots &\vdots  & \ddots& \vdots\\
			0& 0 & \ldots  &\ldots & 0 \\
		\end{pmatrix},
	\end{equation}    and then we construct the $n\times n$ matrix given by
\begin{equation}\label{Eq::F}
		F= \begin{pmatrix}
			\Lambda_{p}& \bf{0_{p\times (n-p)}} \\
                 &   \\
            L &  Z\\
		\end{pmatrix}.
	\end{equation}
 Observe that $F$ has each row sum equals to zero and obviously its spectrum is $\sigma(F)=(0,\lambda_2,\ldots,\lambda_n)$. Next, we make use of Brauer's theorem on $F$ to obtain an $n\times n$ matrix $D$ whose each row and column sum is zero and which is co-spectral to $F$. First, recalling that   the $i$-th column sum  of $F$ is denoted by $s_i= s_i(F)$,  then  it is easy to check that all column sums of $F$ are given by:

  \begin{equation*}
\left\{\begin{aligned}
s_1&=-\lambda_2-x_{p+1}-y_{p+1} & \\
s_2&=\lambda_2-\lambda_3-x_{p+1}+y_{p+1} \\
s_3&=\lambda_3-\lambda_4 &\\
 \vdots  & \\
 s_{p-1}&=\lambda_{p-1}-\lambda_{p} \\
s_{p}&=\lambda_{p}\\
s_{p+1}&=x_{p+1}+y_{p+1}-x_{p+2}-y_{p+2}\\
s_{p+2}&=x_{p+1}-y_{p+1}-x_{p+2}+y_{p+2}\\
\vdots \\
s_{n-3}&=x_{\frac{n+p}{2}-1}+y_{\frac{n+p}{2}-1}-x_{\frac{n+p}{2}}-y_{\frac{n+p}{2}}\\
s_{n-2}&=x_{\frac{n+p}{2}-1}-y_{\frac{n+p}{2}-1}-x_{\frac{n+p}{2}}+y_{\frac{n+p}{2}}\\
s_{n-1}&=x_{\frac{n+p}{2}}+y_{\frac{n+p}{2}}\\
s_{n}&=x_{\frac{n+p}{2}}-y_{\frac{n+p}{2}}\\
\end{aligned}\right.
\end{equation*}
and that $s_1+s_2+\cdots+s_n=0$ as each row sum of $F$ is zero.
 Now applying Brauer's theorem on $F$ with  $z= \left(-\frac{s_1}{n},-\frac{s_2}{n},\ldots,-\frac{s_n}{n}\right)^T$, we obtain the matrix $D=F+ez^T$ whose each row and column sum is zero and such that $\sigma(D)=\sigma(F)$ since $z^Te=0$. Next, we examine in details the entries in each column of $D= (d_{ij})$.

\begin{itemize}
\item $d_{21} =\frac{-(n-1)\lambda_2+x_{p+1}+y_{p+1}}{n}$, \ $d_{p+1 \ 1}=\frac{\lambda_2-(n-1)x_{p+1}+y_{p+1}}{n},$\\ $d_{p+2 \ 1}=\frac{\lambda_2+x_{p+1}-(n-1)y_{p+1}}{n}$,  and $d_{i1}=-\frac{s_1}{n}$  for all $ i\notin\{ 2,p+1,p+2\}$.

\item $d_{22}=\frac{(n-1)\lambda_2+\lambda_3+x_{p+1}-y_{p+1}}{n}$, \ $d_{32}=\frac{-\lambda_2-(n-1)\lambda_3+x_{p+1}-y_{p+1}}{n}$, \\
 $d_{p+1 \ 2}=\frac{-\lambda_2+\lambda_3+x_{p+1}+(n-1)y_{p+1}}{n},$
$d_{p+2 \ 2}=\frac{-\lambda_2+\lambda_3-(n-1)x_{p+1}-y_{p+1}}{n}$ and $d_{i2}=-\frac{s_2}{n}$  for all  $i\notin\{ 2,3, p+1,p+2\}$.
\item $d_{33}=\frac{(n-1)\lambda_3+\lambda_4}{n}$, $d_{43}=\frac{-\lambda_3-(n-1)\lambda_4}{n}$, and $d_{i3}=-\frac{s_3}{n}$
 for all  $i\notin\{ 3, 4\}$. \\   \ \ \ \vdots
 \item $d_{p-1 \ p-1}=\frac{(n-1)\lambda_{p-1}+\lambda_p}{n}$, $d_{p \ p-1}=\frac{-\lambda_{p-1}-(n-1)\lambda_p}{n}$, and $d_{i \ p-1}=-\frac{s_{p-1}}{n}$
 for all  $i\notin\{ p-1, p\}$.
 \item $d_{p p}=\frac{(n-1)\lambda_{p}}{n}$,  and $d_{i p}=-\frac{s_{p}}{n}$
 for all  $i\neq p$.
 \item $d_{p+1 \ p+1} =\frac{(n-1)x_{p+1}-y_{p+1}+x_{p+2}+y_{p+2}}{n}$,\\ $d_{p+2 \ p+1}=\frac{-x_{p+1}+(n-1)y_{p+1}+x_{p+2}+y_{p+2}}{n},$ \\                                       $d_{p+3 \ p+1} =\frac{-x_{p+1}-y_{p+1}-(n-1)x_{p+2}+y_{p+2}}{n},$ \\ $d_{p+4 \ p+1} =\frac{-x_{p+1}-y_{p+1}+x_{p+2}-(n-1)y_{p+2}}{n},$   and  \\                                   $d_{i \ p+1}=-\frac{s_{p+1}}{n}$  for all $ i\notin\{ p+1, p+2, p+3, p+4\}$.

     \item $d_{p+1 \ p+2} =\frac{-x_{p+1}-(n-1)y_{p+1}+x_{p+2}-y_{p+2}}{n},$ \\ $d_{p+2 \ p+2}=\frac{(n-1)x_{p+1}+y_{p+1}+x_{p+2}-y_{p+2}}{n},$ \\                             $d_{p+3 \ p+2} =\frac{-x_{p+1}+y_{p+1}+x_{p+2}+(n-1)y_{p+2}}{n},$ \\ $d_{p+4 \ p+1} =\frac{-x_{p+1}+y_{p+1}-(n-1)x_{p+2}-y_{p+2}}{n},$  and  \\                       $d_{i \ p+2}=-\frac{s_{p+2}}{n}$  for all $ i\notin\{ p+1, p+2, p+3, p+4\}$. \\  \ \ \ \vdots

 \item $d_{n-3 \ n-3} =\frac{(n-1)x_{\frac{n+p}{2}-1}-y_{\frac{n+p}{2}-1}+x_{\frac{n+p}{2}}+y_{\frac{n+p}{2}}}{n},$ \\
 $d_{n-2 \ n-3}=\frac{-x_{\frac{n+p}{2}-1}+(n-1)y_{\frac{n+p}{2}-1}+x_{\frac{n+p}{2}}+y_{\frac{n+p}{2}}}{n},$ \\                                                                   $d_{n-1\ n-3} =\frac{-x_{\frac{n+p}{2}-1}-y_{\frac{n+p}{2}-1}-(n-1)x_{\frac{n+p}{2}}+y_{\frac{n+p}{2}}}{n},$ \\                                                               $d_{n \ n-3} =\frac{-x_{\frac{n+p}{2}-1}-y_{\frac{n+p}{2}-1}+x_{\frac{n+p}{2}}-(n-1)y_{\frac{n+p}{2}}}{n},$   and  \\                                                                $d_{i \ n-3}=-\frac{s_{n-3}}{n}$  for all $ i\notin\{ n-3, n-2, n-1, n\}$.

 \item $d_{n-3 \ n-2} =\frac{-x_{\frac{n+p}{2}-1}-(n-1)y_{\frac{n+p}{2}-1}+x_{\frac{n+p}{2}}-y_{\frac{n+p}{2}}}{n},$ \\                                                            $d_{n-2 \ n-2}=\frac{(n-1)x_{\frac{n+p}{2}-1}+y_{\frac{n+p}{2}-1}+x_{\frac{n+p}{2}}-y_{\frac{n+p}{2}}}{n},$ \\                                                             $d_{n-1 \ n-2} =\frac{-x_{\frac{n+p}{2}-1}+y_{\frac{n+p}{2}-1}+x_{\frac{n+p}{2}}+(n-1)y_{\frac{n+p}{2}}}{n},$ \\                                                                              $d_{n \ n-2} =\frac{-x_{\frac{n+p}{2}-1}+y_{\frac{n+p}{2}-1}-(n-1)x_{\frac{n+p}{2}}-y_{\frac{n+p}{2}}}{n},$   and  \\                                                                                $d_{i \ n-2}=-\frac{s_{n-2}}{n}$  for all $ i\notin\{ n-3, n-2, n-1, n\}$.

\item  $d_{n-1 \ n-1} =\frac{(n-1)x_{\frac{n+p}{2}}-y_{\frac{n+p}{2}}}{n},$ \                                                                                                   $d_{n \ n-1} =\frac{-x_{\frac{n+p}{2}}+(n-1)y_{\frac{n+p}{2}}}{n},$ \\  and                                                                                                $d_{i \ n-1}=-\frac{s_{n-1}}{n}$  for all $ i\notin\{ n-1, n\}$.

\item  $d_{n-1 \ n} =\frac{-x_{\frac{n+p}{2}}-(n-1)y_{\frac{n+p}{2}}}{n},$ \                                                                                                   $d_{n n} =\frac{(n-1)x_{\frac{n+p}{2}}+y_{\frac{n+p}{2}}}{n},$ \\  and                                                                                                     $d_{i n}=-\frac{s_{n}}{n}$  for all $ i\notin\{ n-1, n\}$.
\end{itemize} As before, if we let $w= \min\{\pm \lambda_2,\ldots,\pm \lambda_p,\pm x_{p+1}, \pm y_{p+1} ,\ldots, \pm x_\frac{n+p}{2}, \pm y_\frac{n+p}{2} \},$ then a simple check on the entries of $D$, shows  that in order to  obtain from $D$ a nonnegative generalized doubly stochastic matrix  via the use of Brauer's theorem, we need to add to each entry of $D$,
$\frac{(n-1) |w|+|w|+|w|+|w|}{n}=\frac{(n+2) |w|}{n}.$ In other words, we apply Brauer's theorem on $D$ with
 $z=\left(\frac{(n+2) |w|}{n},\ldots, \frac{(n+2) |w|}{n}\right)^T$ to obtain a nonnegative generalized doubly stochastic matrix $M=D+ez^T$ with spectrum
  $((n+2) |w|; \lambda_2,\ldots,\lambda_n ).$\\
\underline{Case 2:} \\
 Suppose now that all entries in the list $\left(\lambda_2,\ldots, \lambda_n\right)$ are  non-real.  In this case $n$ has to be odd with $n=2s-1$ and for our purposes, we consider the  following  $n\times n$ matrix $F'$ given by:
\begin{equation}\label{Eq::F'}
		F'= \begin{pmatrix}
            0 &0 & 0 &\ldots &0 &\ldots & & & 0\\
			-x_{2}+y_{2} & x_{2} &-y_{2} &\ddots & \ddots & \ddots & & & \vdots\\
           -x_{2}-y_{2} & y_{2} &x_{2} & 0 & 0 &\ddots &\ddots & & \vdots\\
            0 & -x_{3} & y_{3} & x_{3} &-y_{3} & 0 & \ddots\\
           0 & -y_{3} &-x_{3} &  y_{3} &x_{3} &0 &\ddots & & \vdots\\
			0  &0 & 0 &\ddots & \ddots& \ddots&\ddots& 0&0\\
            \vdots &\vdots &\vdots &\ddots & \ddots& \ddots&\ddots& 0 &0\\
			 \vdots  & \vdots & \vdots & & & -x_s &  y_s  & x_s & -y_s\\
			0 & \ldots & \ldots &\ldots & \ \ \ 0 &-y_s & -x_s &  y_s & x_s\\
		\end{pmatrix}.
	\end{equation}
Now for $i=1,\ldots,n$, it is straightforward  to find the $i$-th column sum  $s_i(F')$ of $F'$. Indeed,  a quick look shows that all column sums of $F'$ are given by:

  \begin{equation*}
\left\{\begin{aligned}
s_1(F')&=-2x_{2} \\
s_2(F')&=x_{2}+y_{2}-x_{3}-y_{3} \\
s_3(F')&=x_{2}-y_{2}-x_{3}+y_{3} \\
 \vdots  & \\
s_{n-3}(F')&=x_{s-1}+y_{s-1}-x_{s}-y_{s} \\
s_{n-2}(F')&=x_{s-1}-y_{s-1}-x_{s}+y_{s} \\
s_{n-1}(F')&=x_{s}+y_{s}\\
s_{n}(F')&=x_{s}-y_{s}\\
\end{aligned}\right.
\end{equation*}
Similar to what we have done earlier, we first use Brauer's theorem on $F'$ with $z= \left(-\frac{s_1(F')}{n},-\frac{s_2(F')}{n},\ldots,-\frac{s_n(F')}{n}\right)^T$ to obtain the matrix $Q=F'+ez^T$ whose each row and column sum is zero and such that $\sigma(Q)=\sigma(F')$ since $z^Te=0$. By examining in details the entries in each column of $Q$, we see that  as in the previous case, we need to add $\frac{(n+2) |w|}{n}$  to each entry of $Q$ to obtain a nonnegative matrix where in this case $w= \min\{\pm x_{2}, \pm y_{2} ,\ldots, \pm x_{s}, \pm y_{s} \}$. In other words, we apply Brauer's theorem on $Q$ with $z=\left(\frac{(n+2) |w|}{n},\ldots, \frac{(n+2) |w|}{n}\right)^T$ to obtain a nonnegative generalized doubly stochastic matrix $N=Q+ez^T$ with spectrum $((n+2) |w|; \lambda_2,\ldots,\lambda_n ).$
\end{proof}
\subsection{ Particular cases for improving the upper bound for the doubly stochastic realization }
Our  objective in this subsection is to improve the upper bound $(n+2) |w|$ found in the preceding theorem in some particular cases. First, it is worthy to mention that obtaining this bound is mainly credited to the fact that every column in the matrices $F$ and $F'$ used above has at most 4 nonzero entries. So in order to improve this bound through the use of our method, we need to start with a matrix $T$ that has the same properties provided that each column has at most 3 or less nonzero entries. Now we study 3 particular situations where this happens.

 Suppose that in the $n\times n$ matrix $F$ above defined by \begin{equation}
		F= \begin{pmatrix}
			\Lambda_{p}& \bf{0_{p\times (n-p)}} \\
                 &   \\
            L &  Z\\
		\end{pmatrix},
	\end{equation}  the number $p-1$ of real eigenvalues  which are exhibited in $\Lambda$ is bigger than $n-p-2$. Then, we can easily replace the matrix $F$ with another matrix $T$ whose each row sum is zero and with $\sigma(F)=\sigma(T)$ and such that each column of $T$ has at most 3 nonzero entries. Indeed, Let \begin{equation}
		T= \begin{pmatrix}
			\Lambda_{p}& \bf{0_{p\times (n-p)}} \\
                 &   \\
            L' &  Z'\\
		\end{pmatrix}.
	\end{equation}
where $L'$  and $Z'$ are obtained from $L$ and $Z$ by exchanging some entries  in such a way that each column of $T$ has at most 3 possible nonzero entries. More explicitly, in  $L$ we first exchange  the $(2,2)$-entry with the $(2,p)$-entry and then we move the last possible nonzero entry in each  of the first  $(n-p-2)$ columns of $Z$ to a convenient position in $L$ while keeping it in the same row (note that the last two columns in $Z$ have only two possible nonzero entries). Indeed,  $Z'$ and $L'$ are given by:

  {\tiny \begin{equation}
		Z'= \begin{pmatrix}
			x_{p+1}& -y_{p+1} & 0 & 0 & 0 & 0 &\ldots & & & 0\\
           y_{p+1}& x_{p+1} &0 & 0 &\ddots &0 &\ldots & & & \vdots\\
			-x_{p+2}& y_{p+2} & x_{p+2} &-y_{p+2} &\ddots & \ddots & \ddots & & & \vdots\\
           0& 0 & y_{p+2} &x_{p+2} & 0 & 0 &\ddots &\ddots & & \vdots\\
           0& 0 & -x_{p+3} & y_{p+3} & x_{p+3} &-y_{p+3} & 0 & \ddots\\
           0& 0 & 0 & 0 &  y_{p+3} &x_{p+3} &0 &\ddots & & \vdots\\
			0&0  &0 & 0 &\ddots & \ddots& \ddots&\ddots& 0&0\\
            \vdots&\vdots &\vdots &\vdots &\ddots & \ddots& \ddots&\ddots& 0 &0\\
			0 & 0  & \vdots & \vdots & & & -x_\frac{n+p}{2} &  y_\frac{n+p}{2}\  & x_\frac{n+p}{2} & -y_\frac{n+p}{2}\\
\\
			0& 0 & \ldots & \ldots &\ldots & 0 & 0 & 0 &  y_\frac{n+p}{2} & x_\frac{n+p}{2}\\
		\end{pmatrix},\ \ \ \ \ \ \ \ \ \  \ \ \
	\end{equation}}   which is a $(n-p)\times (n-p)$ matrix, and

{\scriptsize  \begin{equation}
		L'= \begin{pmatrix}
			-x_{p+1}& y_{p+1} & 0 &\ldots & \ldots &\ldots & \ldots &\ldots &0&0\\
           -y_{p+1}& 0 &0 &0  &\ddots & 0 &\ldots & \ldots &0 &-x_{p+1}\\
           \\
           0& 0& -x_{p+2}& -y_{p+2} & 0 &  0 &\ldots &\dots &\dots &0 \\
           0& 0& 0& 0 & -y_{p+3} &-x_{p+3} &  0 &\ldots &\ldots &\vdots \\
			0&0  &0  &\ddots & \ddots & \ddots& \ddots &\ddots &\ddots &\vdots\\
            \vdots&\vdots &\ddots & \ddots & \ddots& \ddots& \ddots& \ddots& \ddots&  \vdots\\
          0 & 0  & \vdots & \vdots & 0& 0& -x_{\frac{n+p}{2}-1} &  -y_{\frac{n+p}{2}-1}\  & 0 & 0\\
          \\
          0& 0 & \ldots & \ldots &0 & 0 &0&0 &-y_\frac{n+p}{2} & -x_\frac{n+p}{2} \\
		\end{pmatrix}, \ \ \ \ \ \ \ \ \
	\end{equation}}  which is an $(n-p)\times p $ matrix.  As mentioned earlier, this construction of $T$ insures that each column has at most 3 possible nonzero
entries. As it was applied on $F$ twice, similar applications of  Brauer's theorem on $T$  where for the second application of this theorem we choose
$z=\left(\frac{(n+1) |w|}{n},\ldots, \frac{(n+1) |w|}{n}\right)^T$. Thus, we conclude the following.
\begin{theorem} If $\left(\lambda_2,\ldots, \lambda_n\right)$ is any list of complex numbers which is closed under complex conjugation
 with $\lambda_2\geq \cdots\geq \lambda_p   $  where $n-p-1\le p\leq n$ and let $\lambda_{p+1},\ldots,\lambda_n$ be non-real with their real parts and imaginary parts are respectively denoted by $x_i=\Re(\lambda_i)$ and  $y_i=\Im(\lambda_i)$ for $i=p+1,\ldots,\frac{n+p}{2}$. Moreover, let
  $$w= \min\{\pm \lambda_2,\ldots,\pm \lambda_p,\pm x_{p+1}, \pm y_{p+1} ,\ldots, \pm x_\frac{n+p}{2}, \pm y_\frac{n+p}{2} \}.$$  Then, there exists a least real number $\lambda_0$
with $\max\limits_{2\leq j\leq n}|\lambda_j|\leq\lambda_0\leq (n+1)|w|$
such that the list of complex numbers $( \lambda_1;\lambda_2,\ldots,\lambda_n)$ is realizable by an $n\times n$ nonnegative generalized doubly stochastic matrix  if and only if $\lambda_1\geq \lambda_0$.
\end{theorem}

The second particular situation that we are interested  in, is  when all $\left(\lambda_2,\ldots, \lambda_n\right)$  are either real or pure imaginary. For the first case, we shall use the matrix $\Lambda_n$ where obviously each column has at most 2 possible nonzero entries. For the second case, we deal with the matrix $F'$ whose each column also has at most 2 possible nonzero entries as all real parts are zeroes. Repeating the same process as earlier, the bound is now improved to
 $n|w|$. As a conclusion, we have the following result.
\begin{theorem} If $\left(\lambda_2,\ldots, \lambda_n\right)$ is any list of complex numbers which is closed under complex conjugation
 such that all of them are either real or all of them are pure imaginary and let
  $$w= \min\{\pm \Re(\lambda_2),\pm \Im(\lambda_2),\ldots,\pm \Re(\lambda_n),\pm \Im(\lambda_n) \}.$$  Then, there exists a least real number $\lambda_0$
with $\max\limits_{2\leq j\leq n}|\lambda_j|\leq\lambda_0\leq n|w|$
such that the list of complex numbers $(\lambda_1;\lambda_2,\ldots,\lambda_n)$ is realizable by an $n\times n$ nonnegative generalized doubly stochastic matrix  if and only if $\lambda_1\geq \lambda_0$.
\end{theorem}

 We conclude this section by noting that it is possible to lower the bound more by imposing more restrictions on $\left(\lambda_2,\ldots, \lambda_n\right)$. Indeed, assume that all of them are \emph{non-positive} and further without loss of generality suppose  that $0\geq \lambda_2\geq\cdots\geq \lambda_n$, then  we consider the following $n\times n$ matrix $\Gamma$ whose each row sum is zero and which is given by:
 \begin{equation}
		\Gamma=\begin{pmatrix}
			0& 0&  \ldots & \ldots & 0 \\
			-\lambda_2&\lambda_2& 0 &\ldots & 0\\
            -\lambda_3& 0 & \lambda_3  & \ddots & \vdots\\
            \vdots & \vdots & \ddots & \ddots &  0 \\
			-\lambda_n & 0& \ldots &   0 & \lambda_n \\
			
		\end{pmatrix},
	\end{equation}
 Repeating  again the same process of applying Brauer's theorem this time on the matrix  $\Gamma$ with the chosen vector
 $z= \left(-\frac{s_1(\Gamma)}{n},-\frac{s_2(\Gamma)}{n},\ldots,-\frac{s_n(\Gamma)}{n}\right)^T$, we obtain the matrix $\Theta=\Gamma +ez^T$ which is given by:

 \begin{equation}
		\Theta=\begin{pmatrix}
\frac{\lambda_2+\cdots+\lambda_n}{n}& \frac{-\lambda_2}{n} \ \ \ & \frac{-\lambda_3}{n} \ \ \ & \ldots & \ldots & \frac{-\lambda_n}{n}\\
-\frac{(n-1)}{n}\lambda_2+\frac{\lambda_3+\cdots+\lambda_n}{n}&\frac{(n-1)\lambda_2}{n}& \frac{-\lambda_3}{n} \ \ \ & \vdots &\ldots & \vdots\\
 -\frac{(n-1)}{n}\lambda_3+\frac{\lambda_2+\lambda_4+\cdots+\lambda_n}{n}& \frac{-\lambda_2}{n} \ \ \ &\frac{(n-1)\lambda_3}{n} & \ \ \ \ \  \ \ddots & & \vdots\\
 \vdots & \vdots & \vdots & \ \ \ \ \ \ \ddots & & \frac{-\lambda_n}{n} \ \ \  \\
-\frac{(n-1)}{n}\lambda_n+\frac{\lambda_2+\cdots+\lambda_{n-1}}{n}& \frac{-\lambda_2}{n} \ \ \ &\frac{-\lambda_3}{n} \ \ \ &  &   \ldots & \frac{(n-1)\lambda_n}{n} \ \ \\
\end{pmatrix},
\end{equation}
 whose each row and column sum is zero. A simple check shows that the least negative entry in $\Theta$ is $\frac{(n-1)\lambda_n}{n} $ so that when applying Brauer's theorem on $\Theta$ with $z= \left(\frac{(n-1)|\lambda_n|}{n},\cdots,\frac{(n-1)|\lambda_n|}{n}\right)^T$, we obtain a doubly stochastic realization for $\left(\lambda_2,\ldots, \lambda_n\right)$ given by $\Theta +ez^T$ whose spectrum is given by  $\sigma(\Theta +ez^T)=\left((n-1)|\lambda_n|; \lambda_2,\ldots, \lambda_n\right).$ Thus, we have the following conclusion.
 \begin{theorem} Let $\left(\lambda_2,\ldots, \lambda_n\right)$ be any list of non-positive numbers such that $0\geq \lambda_2\geq\cdots\geq \lambda_n$.
   Then, there exists a least real number $\lambda_0$ with $\max\limits_{2\leq j\leq n}|\lambda_j|\leq\lambda_0\leq (n-1)|\lambda_n|$
such that the list of complex numbers $(\lambda_1;\lambda_2,\ldots,\lambda_n)$ is realizable by an $n\times n$ nonnegative generalized doubly stochastic matrix  if and only if $\lambda_1\geq \lambda_0$.
\end{theorem}

\section*{Acknowledgements}

The authors sincerely thank the reviewers for many helpful comments and useful suggestions. Also, many thanks go  the  handling editor  and the editor-in-chief  for their valuable suggestions.
\section*{Declaration of Competing Interest}
The authors declare that they have no competing interests.

\end{document}